\UseRawInputEncoding
\documentclass[11pt]{amsart}
\usepackage{amsmath, amssymb, amsthm, amsfonts, mathrsfs}

\newtheorem{theorem}{Theorem}[section]

\newtheorem{corollary}[theorem]{Corollary}

\theoremstyle{definition}
\newtheorem{definition}{Definition}[section]

\DeclareSymbolFont{AMSb}{U}{msb}{m}{n}
\DeclareMathSymbol{\N}{\mathbin}{AMSb}{"4E}
\DeclareMathSymbol{\Z}{\mathbin}{AMSb}{"5A}
\DeclareMathSymbol{\R}{\mathbin}{AMSb}{"52}
\DeclareMathSymbol{\Q}{\mathbin}{AMSb}{"51}
\DeclareMathSymbol{\I}{\mathbin}{AMSb}{"49}
\DeclareMathSymbol{\C}{\mathbin}{AMSb}{"43}

\theoremstyle{definition}

\begin{document}
\title{Rank-one transformations, odometers, and finite factors}
\author[M. Foreman, S. Gao, A. Hill, C.E. Silva, B. Weiss]{Matthew Foreman, Su Gao, Aaron Hill, Cesar E. Silva, Benjamin Weiss}
\address{Mathematics Department, UC Irvine, Irvine, CA 92697, USA}
\email{mforeman@math.uci.edu}
\address{Department of Mathematics, University of North Texas, 1155 Union Circle \#311430, Denton, TX 76203, USA}
\email{sgao@unt.edu}
%\address{Department of Mathematics, University of Louisville, 328 Natural Sciences Building, Louisville, KY 40292 USA}
%\email{aaronthill@gmail.com}
\address{Proof School, 973 Mission Street, San Francisco, CA, 94103, USA}
\email{ahill@proofschool.org}
\address{Department of Mathematics and Statistics, Williams College,, Williamstown, MA 01267, USA}
\email{csilva@williams.edu}
\address{Institute of Mathematics, Hebrew University of Jerusalem, Jerusalem, Israel}
\email{weiss@math.huji.ac.il}

\date{\today}
\subjclass[2010]{Primary 37A05, 37A35}
\keywords{rank-one transformation, odometer, factor, isomorphism, totally ergodic}
%\thanks{}

\begin{abstract}
In this paper we give explicit characterizations, based on the cutting and spacer parameters, of (a) which rank-one transformations factor onto a given finite cyclic permutation, (b) which rank-one transformations factor onto a given odometer, and (c) which rank-one transformations are isomorphic to a given odometer.  These naturally yield characterizations of (d) which rank-one transformations factor onto some (unspecified) finite cyclic permutation, (d$^\prime$) which rank-one transformations are totally ergodic, (e) which rank-one transformations factor onto some (unspecified) odometer, and (f) which rank-one transformations are isomorphic to some (unspecified) odometer.
\end{abstract}

\maketitle \thispagestyle{empty}

%\tableofcontents

\section{Introduction}

The ultimate motivation of the work done in this paper is the isomorphism problem in ergodic theory as formulated by von Neumann in his seminal paper \cite{vN} of 1932. There he asked for an explicit process to determine when two measure-preserving transformations are measure-theoretically isomorphic. Two important theorems in this direction are von Neumann's theorem classifying discrete spectrum transformations by their eigenvalues, and Ornstein's theorem classifying Bernoulli transformations by their entropy. To our knowledge, no other complete isomorphism invariants that classify a class of transformations  have been found, though of course notions such as mixing, weak mixing, etc., are invariant under isomorphism.  In \cite{FRW}, Foreman, Rudolph, and Weiss showed that the isomorphism relation on the class of all ergodic transformations is complete analytic, in particular not Borel. In some sense, this brings a negative conclusion to the von Neumann program. However, in \cite{FRW} the authors also showed that the isomorphism problem is Borel on the generic class of (finite measure-preserving) rank-one transformations. Thus this provides hope that there should exist some explicit method for determining whether two rank-one transformations are isomorphic. In particular, if one is given a specific rank-one transformation, there should be an explicit description of all rank-one transformations that are isomorphic to it. In this paper we give such explicit descriptions, provided that the given rank-one transformation is an odometer. All the transformations we consider in this paper are invertible finite measure-preserving transformations. 

Another reason for considering odometers is the role they played in a question of Ferenczi. In his survey article \cite{Fe}, Ferenczi asked whether every odometer is isomorphic to a symbolic rank-one transformation.  This question is connected to whether two common definitions of rank-one---the constructive geometric definition and the constructive symbolic definition---are equivalent. As  noted by the referee,  in the Introduction to  Adams--Ferenczi--Petersen \cite{AFP}, the authors mention how one can use Remark 2.10 in Danilenko \cite{D16} to answer this question in the affirmative, and also show how to construct a symbolic rank-one transformation that is isomorphic to any given odometer. The results in this paper can be thought of as a continuation of work in   \cite{AFP}, \cite{D16}. Namely, we explicitly describe {\em all} rank-one transformations that are isomorphic to any given odometer (Theorem~\ref{isomorphictothisodometer}).  In addition, we also explicitly describe all rank-one transformations that are isomorphic to some (unspecified) odometer (Theorem~\ref{isomorphictosomeodometer}). 

Rank-one transformations are determined by two sequences of parameters, known as the cutting parameter and spacer parameter (see Section~\ref{Pre} for the precise definitions). In this paper we give explicit descriptions, in terms of the cutting parameter and spacer parameter, of when a rank-one transformation factors onto a given finite cyclic transformation, or factors onto an (infinite) odometer, or is isomorphic to a given odometer. 

Note that a measure-preserving transformation factors onto a non-trivial finite cyclic transformation if and only if it is not totally ergodic. Thus results in this paper give an explicit description of when an arbitrary rank-one transformation is totally ergodic. This generalizes some result of \cite{GH}, where Gao and Hill gave an explicit description of which rank-one transformations with bounded cutting parameter are totally ergodic. 

The rest of paper is organized as follows. In Section~\ref{Pre} we recall the constructive geometric definition and the constructive symbolic definition of rank-one transformations. We also explicitly define odometers and finite cyclic transformations. In Section~\ref{Fin} we give an explicit description of all rank-one transformations that factor onto a given finite cyclic transformation, as well as a description of rank-one transformations that allow a finite factor. In Section~\ref{Odo} we describe all rank-one transformations that factor onto a given odometer. As a corollary, we get a description of all rank-one transformations that factor onto some odometers. Finally, in Section~\ref{Iso} we describe all rank-one transformations that are isomorphic to a given odometer. Again, this gives rise to a description of all rank-one transformations that are isomorphic to some odometer.

\vskip 12pt
{\em Acknowledgment.} The research in this paper was done at the AIM SQuaRE titled {\it The isomorphism problem for rank-one transformations}. The authors would like to acknowledge the American Institute of Mathematics for the support on this research. M.F. acknowledges the US NSF grant DMS-1700143 for support for this research. S.G. acknowledges the US NSF grants DMS-1201290 and DMS-1800323 for the support of his research. Since August 2019, C.S. has been serving as a Program Director in the Division of Mathematical Sciences at the National Science Foundation (NSF), USA, and as a component of this job, he received support from NSF for research, which included work on this paper. Any opinion, findings, and conclusions or recommendations expressed in this material are those of the authors and do not necessarily reflect the views of the National Science Foundation. We we would like to thank the referee for a careful reading and suggestions that shortened our proofs.

\section{Preliminaries}\label{Pre}

\subsection{Measure-preserving transformations}
We will be  concerned with  Lebesgue spaces, which we shall denote by $(X,\mu)$ or $(Y,\nu)$, and typically not mention the $\sigma$-algebra. We shall assume that the measure of the space is 1 and  in most cases, and unless we explicitly specify to the contrary, we will assume our measures to be nonatomic and call the spaces  standard Lebesgue spaces.   A map $\phi:(X,\mu)\to (Y,\nu)$ is {\it measure-preserving} if for all measurable sets $A$, $\phi^{-1}(A)$ is measurable and $\mu(\phi^{-1}(A))=\nu(A)$. A {\it transformation} $T:(X,\mu)\to (X,\mu)$ is a measure-preserving map that is invertible on a set of full measure and whose inverse is measure-preserving.  We will call $(X,\mu,T)$ a measure-preserving system and, by abuse of notation, also a measure-preserving transformation. 

 If $(X, \mu, T)$ and $(Y, \nu, S)$ are measure-preserving transformations, then a {\em factor} map from $T$ to $S$ is a measure-preserving map $\phi:  (X, \mu) \to (Y, \nu)$ such that for $\mu$-almost every $x \in X$, $\phi \circ T (x) = S \circ \phi (x)$. We say that $T$ {\em factors onto} $S$ if there exists a factor map $\phi$ from $(X,\mu,T)$ onto $(Y,\nu,S)$. If $(X, \mu, T)$ and $(Y, \nu, S)$ are measure-preserving transformations, then an {\em isomorphism} between $T$ and $S$ is a factor map $\phi$ from $(X, \mu,T)$ to $(Y, \nu,S)$ 
 that is invertible a.e..  We note here that neither factor maps nor isomorphisms need to be defined on the entire underlying space $(X, \mu)$, only a subset of $X$ of full measure, and that two measure isomorphisms are considered the same if they agree on a set of full measure.

 %We also remark factor map must be an isomorphism if for every $\mu$-measurable set $A \subseteq X$ and every $\epsilon >0$, there is a the $\nu$-measurable set $B \subseteq (Y)$ such that \mu (A \Delta \phi^{-1} (B)) < \epsilon$.  

\subsection{Rank-one transformations} The constructive geometric definition of a rank-one transformation is given below (see e.g., \cite{Fe}). It describes a recursive cutting and stacking process that produces infinitely many Rokhlin towers (or columns)
to approximate the transformation.
\begin{definition} A measure-preserving transformation $T$ on a standard Lebesgue space $(X, \mu)$ is {\it rank-one}
if there exist sequences of positive integers $r_n > 1$, for $n\in\N=\{0, 1, 2, \dots\}$, and nonnegative
integers $s_{n,i}$, for $n\in\N$ and $0 < i \leq r_n$, such that, if $h_n$ is defined by
$$ h_0 = 1; h_{n+1} = r_nh_n +\sum_{0<i\leq r_n}s_{n,i}, $$
then
\begin{equation}\label{r1} \sum^{+\infty}_{n=0} \frac{h_{n+1}-r_nh_n}{h_{n+1}}< +\infty; \end{equation}
and there are subsets of $X$, denoted by $B_n$ for $n\in\N$, by $B_{n,i}$ for $n\in \N$ and $0<i\leq r_n$, and
by $C_{n,i,j}$ for $n\in\N$, $0<i\leq r_n$ and $0<j\leq s_{n,i}$ (if $s_{n,i}= 0$ then there are no
$C_{n,i,j}$), such that for all $n\in\N$:
\begin{itemize}
\item $\{B_{n,i}\,:\,   0 < i \leq r_n\}$ is a partition of $B_n$,
\item the $T^k(B_n)$, $0\leq k < h_n$, are disjoint,
\item $T^{h_n}(B_{n,i}) = C_{n,i,1}$ if $s_{n,i} \neq 0$ and $i \leq r_n$,
\item $T^{h_n}(B_{n,i}) = B_{n,i+1}$ if $s_{n,i} = 0$ and $i < r_n$,
\item $T(C_{n,i,j}) = C_{n,i,j+1}$ if $j < s_{n,i}$,
\item $T(C_{n,i,s_{n,i}}) = B_{n,i+1}$ if $i < r_n$,
\item $B_{n+1} = B_{n,1}$,
\end{itemize}
and the collection $\bigcup_{n=0}^\infty\{B_n, T(B_n), \dots, T^{h_n-1}(B_n)\}$ is dense in the $\sigma$-algebra of all
$\mu$-measurable subsets of $X$.
\end{definition}

Assumption (\ref{r1}) of this definition is equivalent to the finiteness of the measure $\mu$. 
In this definition the sequence $(r_n)$ is called the {\em cutting parameter}, the sets
$C_{n,i,j}$ are called the {\em spacers}, and the doubly-indexed sequence $(s_{n,i})$ is called
the {\em spacer parameter}. For each $n\in\N$, the collection $\{B_n, T(B_n), \dots, T^{h_n-1}(B_n)\}$
gives the {\em stage-$n$ tower}, with $B_n$ as the {\em base} of the tower, and each $T^k(B_n)$,
where $0 \leq k < h_n$, a {\em level} of the tower. The stage-$n$ tower has height $h_n$. At
stage $n+1$, the stage-$n$ tower is cut into $r_n$ many $n$-blocks of equal measure.
Each block has a base $B_{n,i}$ for some $0 < i\leq r_n$ and has height $h_n$. These
$n$-blocks are then stacked up, with spacers inserted in between. At future
stages, these $n$-blocks are further cut into thinner blocks, but they always
have height $h_n$. 

Note that the base of the stage-$m$ tower, $B_m$, is partitioned into $\{ B_{m,i}\,:\,  0<i\leq r_m\}$, where each $B_{m,i}$ is now a level of the stage-$(m+1)$ tower, with $B_{m,1}=B_{m+1}$ being the base of the stage-$(m+1)$ tower. It is clear by induction that for any $n\geq m$, $B_m$ is partitioned into various levels of the stage-$n$ tower. 

We let $I_{m,n}$, for $n\geq  m$,  denote the set of indices for all levels of the stage-$n$ tower that form a partition of $B_m$, i.e.,
$$ I_{m,n}=\{ i\, :\, T^i(B_n)\subseteq B_m, 0\leq i<h_n\}. $$ Note that $B_m=\bigcup_{i\in I_{m,n}}T^i(B_n)$. $I_{m,n}$ is a finite set of natural numbers that can be inductively computed from the cutting and spacer parameters. For example, 
$$ I_{m,m+1}=\{0, h_m+s_{m,1}, 2h_m+s_{m,1}+s_{m,2}, \dots, (r_m-1)h_m+\sum_{0<i<r_m}s_{m,i}\}. $$

We next turn to the constructive symbolic definition of rank-one transformations. This often gives a succinct way to describe 
a concrete rank-one transformation. We will be talking about finite words over the alphabet $\{0,1\}$. Let $F$ be the set of all finite words
over the alphabet $\{0,1\}$ that start with 0. A {\em generating rank-one sequence} is
an infinite sequence $(v_n)$ of finite words in $F$ defined by induction on $n\in\N$:
$$v_0 = 0; v_{n+1} = v_n1^{s_{n,1}}v_n1^{s_{n,2}}\cdots v_n1^{s_{n,r_n}}$$
for some integers $r_n>1$ and non-negative integers $s_{n,i}$ for $0 < i\leq r_n$.
We continue to refer to the sequence $(r_n)$ as the cutting parameter and
the doubly-indexed sequence $(s_{n,i})$ as the spacer parameter. Note that the
cutting and spacer parameters uniquely determine a generating rank-one sequence.
A generating rank-one sequence converges to an infinite rank-one word $V\in \{0,1\}^{\N}$.
We write $V = \lim_{n}v_n$. 

\begin{definition} Given an infinite rank-one word $V$, the {\em symbolic rank-one
system} induced by $V$ is a pair $(X, \sigma)$, where
$$ X = X_V = \{x \in\{0,1\}^\Z\,:\, \mbox{every finite subword of $x$ is a subword of $V$}\}$$
and $\sigma: X \to X$ is the shift map defined by
$$\sigma(x)(k) = x(k + 1)\ \mbox{for all $k\in\Z$}. $$
\end{definition}

Under the same assumption (\ref{r1}) as in the constructive geometric definition, the symbolic rank-one system will carry a unique non-atomic, invariant probability measure. In this case the symbolic rank-one system will be isomorphic to the rank-one transformation that is constructed with the same cutting and spacer parameters. 

%The following paragraph needs some rewriting. Mention Gao-Ziegler result on the topological rank-one systems.

The symbolic definition does not explicitly describe  odometers (see Subsection \ref{cyclic} below for definitions), which   are considered rank-one transformations. This was the motivation of Ferenczi's question in \cite{Fe}  as discussed in the introduction.
%regarding whether every odometer is isomorphic to a symbolic rank-one system, which was recently answered in the affirmative by Adams, Ferenczi, and Petersen \cite{AFP}. % In the next subsection we describe these degenerate cases of rank-one transformations in more details. 
In contrast, we note that in the topological setting, Gao and Ziegler have recently proved in \cite{GZ} that (infinite) odometers are not topologically isomorphic to symbolic rank-one systems (which are called rank-one subshifts in \cite{GZ}).

When we work with a rank-one transformation we will use
both the terminology and the notation in this subsection.

\subsection{Finite cyclic permutations and odometers\label{cyclic}}

Here we precisely describe what we mean by ``finite cyclic permutation" in the context of measure-preserving transformations. If $k\in\N$ with $k>1$ and $n\in\N$, we denote by $[n]_k$ the unique $m\in\N$ with $m<k$ and $n\equiv m\mod k$. For each $k \in \N$ with $k>1$, let $X_k = \{0, 1, \ldots, k-1\}$, let $\mu_k$ be the measure on $X_k$ where each point has measure $1/k$, and let $f_k: X_k \rightarrow X_k$ given by $f_k (i) = [i+1]_k$.  We let $\Z/k\Z$ denote the transformation $(X_k, \mu_k, f_k)$ and refer to such a transformation as a finite cyclic permutation.  These are the sole cases we consider where the measure is atomic, so the measures are defined on atomic Lebesgue probability spaces, and  we will still refer to $(X_k, \mu_k, f_k)$ as a transformation, though it should be clear from the context, such as when we denote a transformation by $T$, when a transformation is defined on a non-atomic space.  It is natural to speak of a factor map from a measure-preserving transformation $T$ to $(X_k, \mu_k, f_k)$, but since $T$ is implicitly defined on a non-atomic space, it is not possible for such a factor map to be an isomorphism. 

Now we describe what we mean by an odometer (see \cite{Do}).  Loosely it can be described as an inverse limit of a coherent sequence of finite cyclic permutations.  To be more precise, suppose we have a sequence $(k_n : n \in \N)$ of positive integers greater than 1 such that for all $n \in \N$, $k_n | k_{n+1}$.  We now define $X$ as the collection of sequences $\alpha = (\alpha_n : n\in \N) \in \Pi_{n \in \N}  \Z / k_n\Z$ such that for all $m,n \in \N$ with $m \leq n$, $[\alpha_n]_{k_m} = \alpha_m$.  There is a natural measure $\mu$ on $X$ satisfying the following:  for all $n \in \Z$ and all $i \in \{0, 1, \ldots, k_n-1\}$ the set $\{\alpha \in X : \alpha_n = i \}$ has measure $1/k_n$.  There is also a natural bijection $f: X \rightarrow X$ defined by $$f(\alpha) = (f_1(\alpha_1), f_2(\alpha_2), \ldots ) = ([\alpha_1 + 1]_{k_1}, [\alpha_2 + 1]_{k_2}, \dots ).$$
A transformation $(X, \mu, f)$ obtained in this way is called an {\it odometer}. For example, if $k_n=2^n$, one obtains  the standard dyadic odometer.  

The following characterization of when two such odometers are isomorphic is well known.  Suppose $(k_n : n \in \N)$ and $(k_n^\prime : n \in \N)$ are sequences of positive integers greater than 1 such that for all $n \in \N$, $k_n | k_{n+1}$ and $k_n^\prime | k_{n+1}^\prime$.  Then the odometers corresponding to these two sequences are isomorphic if and only if
$$ \{m\in \N\,:\, \exists n \in \N\ (m | k_n)\}=\{m\in\N\,:\, \exists n \in \N\ (m | k_n^\prime)\}.$$  

Because of this characterization we often describe an odometer by an infinite collection $K$ of natural numbers that is closed under taking factors.  If one has such a set $K$, then it is easy to produce a sequence $(k_n: n \in \N)$ of integers $>1$ such that $k_n | k_{n+1}$, for all $n \in \N$, and for which $$K = \bigcup_{n \in \N} \{m \in \N : m | k_n\}.$$  Moreover, any choice of such a sequence $(k_n : n \in \N)$ will give rise to the same odometer, up to isomorphism.  We can now let $\mathcal{O}_K$ denote (any) one of the odometers produced by choosing such a sequence $(k_n: n \in \N)$.  There are canonical ways to choose $\mathcal{O}_K$ based on the maximum power of each prime that occurs in $K$, but we will not go into the details of this canonical choice in this paper.  It is worth noting that the characterization in the preceding paragraph guarantees that if $K \neq K^\prime$ are infinite collections of natural numbers that are closed under factors, then $\mathcal{O}_K \not\cong \mathcal{O}_{K^\prime}.$

Here we collect the important facts about $\mathcal{O}_K$ that we will use in this paper.
\begin{enumerate}
\item  For each $k \in K$, then there is a canonical factor map $\pi_k$ from $\mathcal{O}_K$ to $\Z/k\Z$.
\item  For all $k, k^\prime \in K$, with $k | k^\prime$, then for all $x$ in the underlying set of $\mathcal{O}_K$, $\pi_k (x) = [\pi_{k^\prime} (x)]_k$.
\item  The collection of sets $\{ \pi_k^{-1} (i): k \in K, 0 \leq i < k \}$ generates the $\sigma$-algebra on $\mathcal{O}_K$.
\item  If a measure-preserving transformation factors onto $\Z/k\Z$ for all $k \in K$, then it also factors onto $\mathcal{O}_K$.  If, moreover, the fibers of these maps generate the $\sigma$-algebra on $(X, \mu)$, then that factor map is an isomorphism. The argument for this is similar to the construction of the Kronecker factor of a transformation, see e.g. \cite{Qu}.
\end{enumerate}

\subsection{The notion of $\epsilon$-containment}

In this subsection we define a precise notion of almost containment and briefly describe some of its properties;
this is a standard notion in measure theory also called $(1-\epsilon)$-full.   

\begin{definition}
Let $A$ and $B$ be measurable subsets of positive measure of a measure space $(X, \mu)$ and let $\epsilon >0$.  We say that $A$ is {\em $\epsilon$-contained} in $B$, and write $A \subseteq_{\epsilon} B$, provided that $$\frac{\mu (A \setminus B)}{\mu (A)} < \epsilon.$$ Equivalently, we say that $A$ is {\it $(1-\epsilon)$-full} of $B$ if
$\mu(A\cap B)>(1-\epsilon)\mu(A)$.
\end{definition} 

Here are the basic facts we will need; the reader may refer to e.g. \cite{Si}.
\begin{enumerate}
\item  If $A \subseteq_\epsilon B$ and $A$ is partitioned into sets $A_1, A_2, \ldots, A_r$, there is some $i \leq r$ such that $A_i \subseteq_\epsilon B$.
\item  If $A$ is partitioned into sets $A_1, A_2, \ldots, A_r$ and for all $i \leq r$, $A_i \subseteq_\epsilon B$, then $A \subseteq_\epsilon B$.
%Original:  \item  Let $(X, \mu)$ be a standard Lebesgue space and suppose $\{A_i : i \in I\}$ generates the $\sigma$-algebra on $(X, \mu)$.  Then for all $B \subseteq X$ with positive measure and for all $\epsilon >0$, there exists some $i \in I$ such that $A_i \subseteq_\epsilon B$.
\item  Let $(X, \mu, T)$ be a measure-preserving transformation.  If $A \subseteq_\epsilon B$ and $z \in \Z$, the $T^z (A) \subseteq_\epsilon T^z(B)$.
\item  Let $(X, \mu, T)$ be a rank-one transformation.  If $B \subseteq X$ has positive measure, there there is some $n \in \N$ and some $0 \leq i < h_n$ such that $T^i(B_n) \subseteq_{\epsilon} B$.
%This follows from the fact a rank-one transformation with certain cutting and spacer parameters is isomorphic the natural cutting and stacking rank-one transformations with the same parameters--then use the Lebesgue density theorem.
\end{enumerate}

\section{Factoring onto a finite cyclic permutation}\label{Fin}
%What does it mean for a transformation to factor onto a finite cyclic permutation of $k$ elements?  It means that the underlying space $X$ of the transformation can be partitioned into sets $A_0$, $A_1$, \ldots $A_{k-1}$, each of measure $\frac{1}{k}$, such that the transformation $T$ cyclicly permutes $A_0$, $A_1$, \ldots $A_{k-1}$.
It is quite easy to build a rank-one transformation that factors onto a cyclic permutation of $k$ elements.  Simply ensure that for some $N \in \N$, the height of the stage-$N$ tower is a multiple of $k$ and furthermore insist that every time spacers are inserted after stage-$N$ the number of spacers inserted is a multiple of $k$.  If a rank-one transformation is constructed in this way, then one can define, for all $m \geq N$, a function $\pi_m$ which goes from the stage-$m$ tower to $\Z / k\Z$ defined by $\pi_m (x) = [i]_k$, where $x$ belongs to level $i$ of the stage-$m$ tower.  The method of construction guarantees that if $x$ belongs to the stage-$m$ tower and $n \geq m$, then $\pi_m (x) = \pi_n (x)$.  The domains of the functions $\pi_m$ are increasing and their measure goes to one.  Thus, we can define $\pi$ from a full-measure subset of $X$ to $\Z / k\Z$ by $$\pi (x) = \lim_{m \rightarrow \infty} \pi_m (x).$$
This map $\pi$ is clearly a factor map.

The theorem below gives a full characterization of which transformations factor onto a cyclic permutation of $k$ elements.  
\begin{theorem}
\label{finitefactor1}
Let $(X, \mu, T)$ be a rank-one measure-preserving transformation and let $1 < k \in \N$.  The following are equivalent.
\begin{enumerate}
\item[\rm (i)]  $(X, \mu, T)$ factors onto $\Z / k\Z$.
\item[\rm (ii)]  $\forall \eta > 0, \exists N \in \N, \forall n \geq m \geq N, \exists j \in \Z/k\Z$ such that   $$\frac{  |\{i \in I_{m,n} : [i]_k \neq j \}|}{| I_{m,n}|} < \eta.$$
\end{enumerate}  
\end{theorem}

\begin{proof}  

First we will show that (i) implies (ii).  Suppose that $\pi : X \rightarrow \Z/k\Z$ is a factor map.  The fibers $\pi^{-1} (0), \pi^{-1} (1), \pi^{-1} (2), \ldots, \pi^{-1} (k-1)$ are a partition of $X$ into sets of measure $1/k$ such that $T(\pi^{-1} (j)) = \pi^{-1} ([j+1]_k)$, for all $j \in \Z/k\Z$.   Let $\eta >0$ and choose $\epsilon$ smaller than both  $\eta/2$ and $1/2$.  

Since the levels of the towers generate the $\sigma$-algebra of $X$, there exists $N\in\N$ such that for all $n>m\geq N$,  every level of the stage-$n$ tower is $\epsilon$-contained in $\pi^{-1}(j)$ for some $j\in \Z/k\Z$.
Fix $j_0 \in \Z/k\Z$ such that $B_m \subseteq_\epsilon \pi^{-1}(j_0)$.  We claim that among the levels of the stage-$n$ tower that comprise the base of the stage-$m$ tower, the fraction of those  that are $\epsilon$-contained in $ \pi^{-1}(j_0)$ must be at least $1-2\epsilon$.  In other words, letting $I^\prime = \{i \in I_{m,n}: T^i(B_n) \not\subseteq_\epsilon \pi^{-1} (j_0)\}$, we claim that 
\begin{equation}\label{fraction}\frac{|I^\prime|}{|I_{m,n}|} < 2\epsilon.
\end{equation}  
Suppose this is not the case. Since 
$$B_m \setminus \pi^{-1}(j_0) \supseteq \bigcup_{i \in I^\prime} \left(  T^i(B_n) \setminus \pi^{-1}(j_0) \right), \text{we have that} $$
 $$\mu \left( B_m \setminus \pi^{-1}(j_0) \right) \geq  |I^\prime| \cdot \mu(B_n) \cdot (1 - \epsilon) = \frac{ |I^\prime|}{|I_{m,n}|} \cdot \mu(B_m) \cdot (1 - \epsilon).$$
Therefore, $$\frac{\mu \left( B_m \setminus \pi^{-1}(j_0) \right) }{ \mu(B_m)} \geq   \frac{ |I^\prime|}{|I_{m,n}|} \cdot (1 - \epsilon) \geq (2 \epsilon) \cdot (1-\epsilon) > \epsilon,$$
since $\epsilon < 1/2$.  This contradicts the fact that $B_m$ is $\epsilon$-contained in $\pi^{-1}(j_0)$ and completes the proof of \eqref{fraction}.

Since the levels of the stage-$n$ tower that are $\epsilon$-contained in $\pi^{-1}(j_0)$ are all in the same congruence class mod $k$, there is some $j \in \Z / k\Z$ such that $$\frac{  |\{i \in I_{m,n} : [i]_k \neq j \}|}{| I_{m,n}|} < 2 \epsilon < \eta,$$
completing the  proof that (i) implies (ii).

Next we will show that (ii) implies (i).  Assuming (ii) we construct a factor map $\pi : X \rightarrow \Z/ k\Z$.  

For all $\alpha \in \N$, let $\eta_\alpha = \frac{1}{2^{\alpha+2}}$ and use (ii) to produce $N_\alpha \in \N$.  We may assume that the sequence $(N_\alpha : \alpha \in \N)$ is increasing and  that for each $\alpha$, $N_\alpha$ is large enough that the measure of the stage-$N_\alpha$ tower is at least $1 - \frac{1}{2^{\alpha +1}}$.  Now, for each $\alpha \in \N$ we also choose $j_\alpha \in \Z/k\Z$ such that 
$$\frac{  |\{i \in I_{N_\alpha,N_{\alpha+1}} : [i]_k \neq j_\alpha \}|}{| I_{N_\alpha,N_{\alpha+1}}|} < \eta_\alpha .$$   
 
For all $\alpha \in \N$, define a function $\phi_{\alpha}$ from the stage-$N_\alpha$ tower to $\Z/k\Z$ as follows:  If $x$ belongs to level $i$ of the stage-$N_\alpha$ tower, then $\phi_\alpha (x) = [i]_k$.  Since for most $x$ in the base of the $N_\alpha$-tower, $\phi_{\alpha+1} (x) =  j_\alpha$, the reader can verify  that for all $\alpha \in \N$, $$\mu \left( \{x \in \textnormal{dom}(\phi_\alpha): \phi_{\alpha+1} (x) \neq  j_\alpha   \} \right) < \eta_\alpha.$$

%Indeed, notice that for most $x$ in the base of the $N_\alpha$-tower, $\phi_{\alpha+1} (x) =  j_\alpha$.  To be more precise:
%$$\frac{\mu(\{x \in B_{N_\alpha} : \phi_{\alpha+1} (x) \neq j_\alpha \})}{\mu (B_{N_\alpha})}=\frac{  |\{i \in I_{N_\alpha,N_{\alpha+1}} : [i]_k \neq j_\alpha \}|}{| I_{N_\alpha,N_{\alpha+1}}|}< \eta_\alpha.$$  
%This implies that for all $0\leq i< h_{N_{\alpha}}$, we have
%$$\frac{\mu(\{x \in T^i (B_{N_\alpha}) : \phi_{\alpha+1} (x)  \neq [\phi_{\alpha} (x) + j_\alpha]_k \})}{\mu( T^i (B_{N_\alpha}))}< \eta_\alpha.$$ 
%Thus $$\frac{\mu(\{x \in \bigcup_{0 \leq i < h_{N_{\alpha}}}  T^i (B_N{_\alpha}) : \phi_{\alpha+1} (x)  \neq [\phi_{\alpha} (x) + j_\alpha]_k \})}{\mu(\bigcup_{0 \leq i < h_{N_{\alpha}}}  T^i (B_N{_\alpha}))}< \eta_\alpha.$$  This clearly implies that $$\mu \left( \{x \in \textnormal{dom}(\phi_\alpha): \phi_{\alpha+1} (x) \neq [\phi_\alpha (x) + j_\alpha]_k   \} \right) < \eta_\alpha,$$ which completes the proof of our claim.

Now, for each $\alpha \in \N$, we let $J_\alpha = \sum_{\beta < \alpha} j_\beta$.  Also, for each $\alpha \in \N$ we define a function $\pi_\alpha$ from the stage-$N_\alpha$ tower to $\Z/k\Z$ by $\pi_\alpha (x) = [\phi_\alpha (x) - J_\alpha]_k$.  Since  $\phi_\alpha$ and $\pi_\alpha$ have the same domain for all $\alpha \in \N$, and in addition,  if $x \in \textnormal{dom} (\pi_\alpha)$, then $\pi_{\alpha+1} (x)  = \pi_{\alpha} (x)$ if and only if $\phi_{\alpha+1} (x)  = [\phi_{\alpha} (x) + j_\alpha]_k$, and we already know that  
$\mu \left( \{x \in \textnormal{dom}(\phi_\alpha): \phi_{\alpha+1} (x) \neq [\phi_\alpha (x) + j_\alpha]_k   \} \right) < \eta_\alpha,$ then one can verify  that for all $\alpha \in \N$, 
$$\mu \left( \{x \in \textnormal{dom}(\pi_\alpha): \textnormal{ for all $\beta \geq \alpha$, } \pi_\alpha (x)  = \pi_{\beta} (x)  \} \right) \geq 1 - \frac{1}{2^\alpha }.$$

%Indeed, note that $\phi_\alpha$ and $\pi_\alpha$ have the same domain for all $\alpha \in \N$.  In addition, note that if $x \in \textnormal{dom} (\pi_\alpha)$, then $\pi_{\alpha+1} (x)  = \pi_{\alpha} (x)$ if and only if $\phi_{\alpha+1} (x)  = [\phi_{\alpha} (x) + j_\alpha]_k$.
%We already know that  $$\mu \left( \{x \in \textnormal{dom}(\phi_\alpha): \phi_{\alpha+1} (x) \neq [\phi_\alpha (x) + j_\alpha]_k   \} \right) < \eta_\alpha.$$  
%This implies that $$\mu \left( \{x \in \textnormal{dom}(\pi_\alpha): \pi_{\alpha+1} (x) \neq \pi_\alpha (x)    \} \right) < \eta_\alpha = \frac{1}{2^{\alpha+2}}.$$ And this, in turn, implies that  $$\mu \left( \{x \in \textnormal{dom}(\pi_\alpha): \exists \beta>\alpha \textnormal{ such that }\pi_{\beta} (x) \neq \pi_\alpha (x)    \} \right) <  \frac{1}{2^{\alpha+1}}.$$ Now, since we assumed that the measure of the stage-$N_\alpha$ tower is at least $1 -  \frac{1}{2^{\alpha+1}}$, we have that  $$\mu \left( \{x \in \textnormal{dom}(\pi_\alpha): \textnormal{ for all $\beta \geq \alpha$, } \pi_\alpha (x)  = \pi_{\beta} (x)  \} \right) \geq 1 - \frac{1}{2^\alpha }.$$ This completes the proof of our claim.

It follows that for $\mu$-almost every $x \in X$, the sequence $(\pi_\alpha (x) : \alpha \in \N)$ eventually stabilizes and we can define $$\pi (x) = \lim_{\alpha \rightarrow \infty} \pi_\alpha (x).$$ 

%To conclude that $\pi : X \rightarrow \Z/ k\Z$ is a factor map we need to show that for $\mu$-almost every $x \in x$, $\pi (x)$ and $\pi(Tx)$ are both defined and $\pi (T(x)) = [\pi (x) + 1]_k$.  We know that for $\mu$-almost every $x \in X$:
%\begin{itemize} \item  $\pi (x)$ is defined; thus, for sufficiently large $\alpha \in \N$,  $\pi_\alpha (x) = \pi (x)$.\item  $\pi (T(x))$ is defined;  thus, for sufficiently large $\alpha \in \N$,  $\pi_\alpha (T(x)) = \pi (T(x))$. \item  For sufficiently large $m$, $x$ belongs to the stage-$m$ tower, but not the top level of that tower. \end{itemize}

Choose $\alpha$ sufficiently large so that $\pi_\alpha (x) = \pi(x)$, $\pi_\alpha(T(x)) = \pi(T(x))$ and $x$ belongs to a non-top level of the stage-$N_\alpha$ tower.  If $x$ belongs to level $i$ of the stage $N_\alpha$ tower, then $T(x)$ belongs to level $i+1$ of the stage-$N_\alpha$ tower which implies that $\phi_\alpha (T(x)) = [\phi_\alpha (x) + 1]_k$.  Now, 
$$\pi (T(x)) = \pi_\alpha (T(x)) = [\phi_{\alpha} (T(x)) - J_\alpha]_k  = [\phi_{\alpha} (x) + 1 - J_\alpha]_k =  [\pi (x) +1]_k. $$
Therefore, $\pi: X \rightarrow \Z/k\Z$ is a factor map.
\end{proof}

As a corollary, we obtain a characterization of the rank-one transformations that factor onto some (unspecified) non-trivial finite cyclic permutation, a condition that is well-know to be equivalent to the transformation not being totally ergodic.

\begin{corollary}\label{cortoterg}
Let $(X, \mu, T)$ be a rank-one measure-preserving transformation.  The following are equivalent.
\begin{enumerate}
\item  $T$ factors onto some finite cyclic permutation.
\item  $\exists k \in \N$ with $k>1$, $\forall \eta > 0, \exists N \in \N, \forall n \geq m \geq N, \exists j \in \Z/k\Z$ such that   $$\frac{  |\{i \in I_{m,n} : [i]_k \neq j \}|}{| I_{m,n}|} < \eta.$$
%\item  $T$ is not totally ergodic.
\end{enumerate}  
\end{corollary}

We end with an equivalent characterization as suggested by the referee. The proof is similar to that of Theorem \ref{finitefactor1}. 

\begin{theorem}
\label{finitefactor2}
Let $(X, \mu, T)$ be a rank-one measure-preserving transformation and let $1 < k \in \N$.  The following are equivalent.
\begin{enumerate}
\item[\rm (i)]  $(X, \mu, T)$ factors onto $\Z / k\Z$.
\item[\rm (ii)]  There is an increasing sequence $(q_n)$ 
 such that   $$\sum_{n=1}^\infty \frac{  |\{i \in I_{q_n,q_n+1} : i \equiv 0 \mod k \}|}{| I_{q_n,q_n}|} < \infty.$$
\end{enumerate}  
\end{theorem}

\section{Factoring onto an odometer}\label{Odo}
%Recall that every odometer is isomorphic to an odometer $\mathcal{O}_K$, where $K$ is a infinite subset of natural numbers that is closed under factors.  Moreover, a measure-preserving transformations factors onto $\mathcal{O}_K$ if and only if it factors onto $\Z / k\Z$ for every $k \in K$.  Together with Theorem \ref{finitefactor1}, these immediately imply the following Theorem.

We now give characterizations of which rank-one transformations factor onto a given odometer, and which rank-one transformations factor onto some (unspecified) odometer.  These characterizations are essentially corollaries of Theorem \ref{finitefactor1}.

\begin{theorem}\label{T:factortoodometer}
Let $(X, \mu, T)$ be a rank-one measure-preserving transformation and let $\mathcal{O}_K$ be an odometer.  The following are equivalent.
\begin{enumerate}
\item[\rm (i)]  $(X, \mu, T)$ factors onto $\mathcal{O}_K$.
\item[\rm (ii)]  $\forall k \in K, \forall \eta > 0, \exists N \in \N, \forall n \geq m \geq N, \exists j \in \Z/k\Z$ such that   $$\frac{  |\{i \in I_{m,n} : [i]_k \neq j \}|}{| I_{m,n}|} < \eta.$$
\end{enumerate}  
\end{theorem}

\begin{proof}
Suppose $(X, \mu, T)$ factors onto $\mathcal{O}_K$.  Then for  each $k \in K$, one can compose this factor map with a factor map from $\mathcal{O}_K$ to $\Z/ k\Z$ to get a factor map from $(X, \mu, T)$ to  $\Z/ k\Z$.  Together with Theorem \ref{finitefactor1}, this implies condition (ii).

Now suppose that condition (ii) holds.  By Theorem \ref{finitefactor1} we know that $(X, \mu, T)$ factors onto  $\Z/ k\Z$ for every $k \in K$.  Therefore, $(X, \mu, T)$ factors onto $\mathcal{O}_K$.
\end{proof}

By a  proof is similar to that of Theorem~\ref{T:factortoodometer} we  obtain the following corollary.

\begin{corollary}
Let $(X, \mu, T)$ be a rank-one measure-preserving transformation. The following are equivalent.
\begin{enumerate}
\item[\rm (i)]  $(X, \mu, T)$ factors onto some odometer $\mathcal{O}$.
\item[\rm (ii)]  $\forall M \in \N, \exists k \geq M, \forall \eta > 0, \exists N \in \N, \forall n \geq m \geq N, \exists j \in \Z/k\Z$ such that   $$\frac{  |\{i \in I_{m,n} : [i]_k \neq j \}|}{| I_{m,n}|} < \eta.$$
\end{enumerate}  
\end{corollary}

%If condition (i) holds, then there is some odometer $\mathcal{O}_K$ onto which $(X, \mu, T)$ factors.  For each $k \in K$, one can compose this factor map with a factor map from $\mathcal{O}_K$ to $\Z/ k\Z$ to get a factor map from $(X, \mu, T)$ to  $\Z/ k\Z$.  Since the set $K$ is infinite, Theorem \ref{finitefactor1} implies condition (ii).

%Now suppose that condition (ii) above holds.  Let $K$ be the set of $k \in \N$ for which $(X, \mu, T)$ factors onto $\Z/k\Z$.  Condition (ii) and Theorem \ref{finitefactor1} guarantee that $K$ is infinite.  It is clear that this set is closed under factors.  Therefore, $(X, \mu, T)$ factors onto the odometer $\mathcal{O}_K$.
%\end{proof}

\section{Being isomorphic to a given odometer}\label{Iso}

It turns out that it is not too hard to construct a rank-one transformation that is isomorphic to a given odometer.  Let $K$ be an infinite set of natural numbers that is closed under factors.  First choose a sequence $(k_n: n \in \N)$ of natural numbers such that the factors of the partial products $\prod_{m<n}k_m$ are precisely the set $K$ and for which $$\sum_{n \in \N} \frac{1}{k_n}< \infty. $$
Then build a rank-one transformation by a symbolic construction as follows.  For $n \in \N$, let $v_0 = 0$ and let $v_{n+1} = (v_n)^{k_n-1} 1^{v_n}$.  Then the resulting transformation $T$ is what is called {\em essentially $0$-expansive} by Adams, Ferenczi, and Petersen in \cite{AFP}, and their method shows that $T$ is isomorphic to the odometer $\mathcal{O}_K$. A definition of an isomorphism is also implicit in our results below.

% It is also straightforward to check that $T$ is isomorphic to the odometer $\mathcal{O}_K$. Indeed, note that $S=\mathcal{O}_K$ can be isomorphically constructed by a cutting and stacking process with the cutting parameter $(r_n)$ where $r_n=k_n$. Let $B^S_n$ be the base of the stage-$n$ tower of this construction and let $h_n$ be the height of the stage-$n$ tower. Then in the cutting and stacking construction of $T$ the stage-$n$ tower is also of height $h_n$, and we let $B^T_n$ be the base of the stage-$n$ tower of the construction of $T$. An ismorphism from $T$ to $S$ can be defined as follows. For any $x$ in the domain of $T$, let $N_x\in \N$ be such that $x$ belongs to all stage-$n$ towers for $n\geq N_x$. For each $n\geq N_x$, let $i_{x,n}$ be the level of the stage-$n$ tower such that $x$ belongs, i.e., $x\in T^{i_{x,n}}(B^T_n)$. Then define $\phi(x)$ be the unique element of $$ \bigcap_{n\geq N_x} T^{i_{x,n}}(B^S_n). $$

In this section we characterize in general when a rank-one transformation is isomorphic to a given odometer. The idea is to build on our characterization for rank-one transformations which factor onto a given odometer, and then to examine when a factor map turns out to be an isomorphism. The following result gives the explicit details.

\begin{theorem}
\label{isomorphictothisodometer}
Let $(X, \mu, T)$ be a rank-one measure-preserving transformation and let $\mathcal{O}_K$ be an odometer.  The following are equivalent.
\begin{enumerate}
\item[\rm (I)]  $T$ is isomorphic to $\mathcal{O}_K$.
\item[\rm (II)]  Both of the following hold.
\begin{enumerate}
\item[\rm (IIa)]  $\forall k \in K, \forall \eta > 0, \exists N \in \N, \forall n \geq m \geq N, \exists j \in \Z/k\Z$ such that   $$\frac{  |\{i \in I_{m,n} : [i]_k \neq j \}|}{| I_{m,n}|} < \eta.$$
\item[\rm (IIb)]  $\forall l \in \N, \forall \epsilon>0, \exists k \in K, \exists N \in \N, \forall m \geq N,  \exists D \subseteq \Z / k\Z$ such that $$\frac{|\{i \leq |h_m|: [i]_k \in D\} \Delta I_{l,m} |}{|I_{l,m}|} < \epsilon$$ 
\end{enumerate} 
\end{enumerate} 
\end{theorem}
\begin{proof}  First assume (II).  Using condition (IIa) and the proof of Theorem \ref{finitefactor1} we  construct, for each $k\in K$, a factor map $\pi_k: X \rightarrow \Z/k\Z$.  Recall that $\pi_k$ is built using a series of approximating maps $(\pi_{k, \alpha}: \alpha \in \N)$.

It suffices to show that for every $l \in \N$ and every $\delta > 0$, there is some $k \in K$ and some $E \subseteq \Z/ k\Z$ such that $$\mu (B_l \Delta \pi_k^{-1} [ E] ) < \delta.$$

Let $l \in \N$ and $\delta >0$.  Let $\epsilon = \delta/2$.  First, we use condition (IIb) above to produce $k \in K$ and $N >l$ such that for all $m \geq N$, there exists some $D \subseteq \Z/k\Z$ such that
$$\frac{|\{i \leq |h_m|: [i]_k \in D\} \Delta I_{l,m} |}{|I_{l,m}|} < \epsilon.$$

Since $k \in K$, we have a factor map $\pi_k : X \rightarrow \Z/k\Z$ that is built using the approximating maps $\pi_{k, \alpha}$.  Choose a specific $\alpha \in \N$ so that $\frac{1}{2^{\alpha}} < \delta/2$ and such that $N_\alpha$ is greater than the $N$ produced in the preceding paragraph.  Using the fact that $N_\alpha>N$ and using features of the approximating maps $\pi_{k, \alpha}$ we get the following.

\begin{enumerate}
\item [(i)]  There exists some $D \subseteq \Z/k\Z$ such that
$$\frac{|\{i \leq h_{N_\alpha}: [i]_k \in D\} \Delta I_{l,N_\alpha} |}{|I_{l,N_\alpha}|} < \epsilon.$$

\item [(ii)]  There exists $E \subseteq \Z/k\Z$ such that $$\bigcup_{d \in D}  ( \bigcup_{\substack{0 \leq i < h_{N_\alpha}\\ [i]_k =d}} T^i (B_{N_\alpha})) = \bigcup_{e \in E}  \pi_{k, \alpha}^{-1} (e).  $$

\item [(iii)] $\mu (\{x \in \textnormal{dom}(\pi_{k, \alpha}): \pi_{k, \alpha} (x) = \pi_k(x)\}) \geq 1 - \frac{1}{2^{\alpha}}$.
\end{enumerate}

Using these properties one can show  that $$\mu (B_l \Delta \pi_k^{-1} [ E] ) < \delta,$$
%First, it follows immediately from the definition of $I_{l, N_\alpha}$ that  $$B_l = \bigcup_{i \in I_{l, N_\alpha}} T^i (B_{N_\alpha}). $$ Second, it follows from item (i) above that$$\mu \left(   \left(\bigcup_{i \in I_{l, N_\alpha}} T^i (B_{N_\alpha}) \right)  \Delta \left( \bigcup_{\substack{0 \leq i < h_{N_\alpha}\\  [i]_k \in D}} T^i (B_{N_\alpha}) \right)  \right) < \epsilon.$$
%Third, notice that $$ \left( \bigcup_{\substack{0 \leq i < h_{N_\alpha}\\ [i]_k \in D}} T^i (B_{N_\alpha}) \right) = \bigcup_{d \in D}  \left( \bigcup_{\substack{0 \leq i < h_{N_\alpha}\\ [i]_k =d}} T^i (B_{N_\alpha}) \right).$$ Fourth, it follows from item (ii) above that $$\bigcup_{d \in D}  \left( \bigcup_{\substack{0 \leq i < h_{N_\alpha}\\ [i]_k =d}} T^i (B_{N_\alpha}) \right) = \bigcup_{e \in E} \pi_{k, \alpha}^{-1} (e). $$ Fifth, it follows from item (iii) above we that  $$ \mu \left( \left(  \bigcup_{e \in E}  \pi_{k, \alpha}^{-1} (e) \right) \Delta \left( \bigcup_{e \in E}  \pi_{k}^{-1} (e)  \right)  \right)< \frac{1}{2^\alpha} < \epsilon.$$ Sixth, it is clear that $$\bigcup_{e \in E}  \pi_{k}^{-1} (e) = \pi_k^{-1} (E). $$
%Putting these all together, we get $$\mu (B_l \Delta \pi_k^{-1} (E) ) < \epsilon + \epsilon = \delta.$$
completing the proof that $(X, \mu, T)$ is isomorphic to $\mathcal{O}_K$.

Now we assume that $(X, \mu, T)$ is isomorphic to $\mathcal{O}_K$ and let $\phi$ be an isomorphism between $T$ and $\mathcal{O}_K$.  For each $k \in K$ we can compose $\phi$ with the canonical factor map of $\mathcal{O}_K$ onto $\Z/k\Z$ to get a factor map $\pi_{k}$ from $X$ to $\Z/k\Z$.  For such a $k\in K$, Theorem \ref{finitefactor1} guarantees that $\forall \eta > 0, \exists N \in \N, \forall n \geq m \geq N, \exists j \in \Z/k\Z$ such that   $$\frac{  |\{i \in I_{m,n} : [i]_k \neq j \}|}{| I_{m,n}|} < \eta.$$  Thus we have  condition (IIa).  

Next, exchanging the variable $\epsilon$ for $\delta$ in condition (IIb), we will prove that $\forall l \in \N, \forall \delta>0, \exists k \in K, \exists N \in \N, \forall m \geq N,  \exists D \subseteq \Z / k\Z$ such that $$\frac{|\{i \leq |h_m|: [i]_k \in D\} \Delta I_{l,m} |}{|I_{l,m}|} < \delta.$$
%Note that we have exchanged the variable $\epsilon$ for $\delta$ in condition (IIb).  We want to reserve the variable $\epsilon$ for a certain role in the proof that follows.

Let $l \in \N$ and $\delta>0$. Let $\epsilon = \delta \cdot \mu(B_l)/4$.  The reader can verify  that there exists some $k \in K$ and $E \subseteq \Z/k\Z$ such that 
\begin{equation}
\mu (B_l \Delta \pi_k^{-1} (E)) < \epsilon. \tag{*}
\end{equation} 
 %Indeed since $\phi$ is an isomorphism from $T$ to $\mathcal{O}_K$, we can approximate $B_l$ by the pre-image, under $\phi$, of a basic clopen set $\mathcal{U}$ in the underlying set of $\mathcal{O}_K$, giving $\mu (B_l \Delta \phi^{-1} (\mathcal{U})) < \epsilon$.  Since $\mathcal{U}$ is a basic clopen set in the underlying set of $\mathcal{O}_K$, there is some $k \in K$ and some $E \subseteq \Z/k\Z$ such that $\mathcal{U}$ is the pre-image, under the canonical projection from $\mathcal{O}_K$ onto $\Z/k\Z$, of the set $E$.  Our definition of $\pi_k : X \rightarrow \Z/k\Z$ gives that $\mathcal{U} = \pi_k^{-1} (E)$.  Therefore, $\mu (B_l \Delta \pi_k^{-1} (E)) < \epsilon$, completing our proof of this first claim.

%WE ARE USING THE FOLLOWING FACT:  the topology on $\mathcal{O}$ is generated by basic clopen sets of the form $\mathcal{U}_{k, D}$, where $k$ is one of the finite factors of $\mathcal{O}$ and $D \subseteq \Z/k\Z$.  

We next claim that there exists $N \in \N$ such that for all $m \geq N$ there exists some $j \in \Z/k\Z$ such that for all $0 \leq i < h_m$, $T^{i}(B_m) \subseteq_\epsilon \pi_k^{-1} ([i +j]_k)$.  We can prove this with similar methods. 

 %Indeed, consider the measurable set $\pi_k^{-1} (0)$.  Since the levels of the towers generate the $\sigma$-algebra of measurable sets in $X$, there must be some level of some tower such that is $\epsilon$-contained in $\pi_k^{-1} (0)$.  Let $N \in \N$ and $0 \leq i_N < h_N$ be such that $T^{i_N}(B_N) \subseteq_{\epsilon} \pi_k^{-1} (0)$.  For $m \geq N$, $T^{i_N}(B_N)$ is partitioned into levels of the stage-$m$ tower.  At least one of those levels must be $\epsilon$-contained in $\pi_k^{-1}(0)$.  Let $0 \leq i_m < h_m$ be such that $T^{i_m} (B_m) \subseteq_{\epsilon} \pi_k^{-1}(0)$.  Since $\pi_k$ is a factor map from $X$ to $\Z/k\Z$, it must be the case that for any $z \in \Z$, $T^{i_m + z} (B_m) \subseteq_{\epsilon} \pi_k^{-1}([z]_k)$.  In particular, letting $z = i-i_m$, we have $T^{i}(B_m) \subseteq_{\epsilon} \pi_k^{-1}([i + (-i_m)]_k)$.  Letting $j = [-i_m]_k$ yields $T^{i}(B_m) \subseteq_\epsilon \pi_k^{-1} ([i +j]_k)$. Noting that the definition of $j$ depends on $m$ but not on $i$, the proof of this claim is complete.

Fix such an $N \in \N$ that also satisfies $\mu\left(\bigcup_{0\leq i<h_N}T^i(B_N)\right)  >1-\epsilon$ and let $m \geq N$.  We now claim that there exists $D \subseteq \Z/k\Z$ such that 
\begin{equation} \mu (  \bigcup_{\substack{0 \leq i < h_m\\ [i]_k \in D}} T^i(B_m)   \Delta \ \pi_k^{-1} (E) ) < 3 \epsilon. \tag{**} \end{equation}

Combining equations (*) and (**) we now have that 
$$\mu (  \bigcup_{\substack{0 \leq i < h_m\\ [i]_k \in D} }T^i(B_n)  \Delta \  B_l) < 4 \epsilon.$$
To finish the proof of the theorem, 
%we need to show that  $$\frac{|\{i < h_m: [i]_k \in D\} \Delta I_{l,m} |}{|I_{l,m}|}  <\delta.$$
note that
$$\begin{array}{l} \displaystyle\frac{|\{i < h_m: [i]_k \in D\} \Delta I_{l,m} |}{|I_{l,m}|} 
=\frac{\displaystyle \mu (  \bigcup_{\substack{0 \leq i < h_m\\ [i]_k \in D}} T^i(B_m)  \Delta   \bigcup_{i \in I_{l,m} } T^i(B_m) ) }{\displaystyle \mu  (   \bigcup_{i \in I_{l,m} } T^i(B_m) )} \\
=\frac{\displaystyle \mu (   \bigcup_{\substack{0 \leq i < h_m\\ [i]_k \in D}} T^i(B_m)  \Delta   B_l ) }{\displaystyle \mu  \left( B_l   \right)} <\displaystyle\frac{4\epsilon}{\mu(B_l)} = \delta.\end{array}$$
\end{proof}

Next we characterize when a rank-one transformation is isomorphic to some (unspecified) odometer.

\begin{theorem}
\label{isomorphictosomeodometer}
Let $(X, \mu, T)$ be a rank-one measure-preserving transformation.  The following are equivalent.
\begin{enumerate}
\item[\rm (I)]  $T$ is isomorphic to an odometer.
\item[\rm (II)]  For all $l \in \N$ and all $\epsilon>0$, there is some $k \in \N$ such that for all $\eta >0$ there exists an $N \in \N$ such that for all $n > m \geq N$, 
\begin{enumerate}
\item[\rm (IIa)]  There is some $j \in \Z / k\Z$ such that $$\frac{|\{i \in I_{m,n}: [i]_k \neq j\}|}{|I_{m,n}|} < \eta$$
\item[\rm (IIb)]  There is some $D \subseteq \Z / k\Z$ such that $$\frac{|\{i \leq |h_m|: [i]_k \in D\} \Delta I_{l,m} |}{|I_{l,m}|} < \epsilon$$
\end{enumerate}
\end{enumerate}  
\end{theorem}

\begin{proof}
%We first show that (I) implies (II).  This will be quite easy to do, using Theorem \ref{isomorphictothisodometer}.
  Suppose $T$ is isomorphic to an odometer.  Let $K$ be the finite factors of that odometer.  Let $l \in \N$ and $\epsilon >0$.  Using condition (IIb) of Theorem \ref{isomorphictothisodometer} we can find some $k \in K$ and some $N_1 \in \N$, such that $\forall m \geq N_1,  \exists D \subseteq \Z / k\Z$ such that $$\frac{|\{i \leq |h_m|: [i]_k \in D\} \Delta I_{l,m} |}{|I_{l,m}|} < \epsilon$$ 
For any $\eta >0$ we can use that specific $k\in K$ and condition (IIa) of Theorem \ref{isomorphictothisodometer} to find $N_2 \in \N$ such that $\forall n \geq m \geq N_2, \exists j \in \Z/k\Z$ such that   $$\frac{  |\{i \in I_{m,n} : [i]_k \neq j \}|}{| I_{m,n}|} < \eta.$$
Letting $N = \max\{N_1, N_2\}$ we complete condition (II) of the theorem.
%have that for all $n \geq m \geq N$, 
%\begin{enumerate}
%\item[\rm (i)] There is some $j \in \Z / k\Z$ such that $$\frac{|\{i \in I_{m,n}: [i]_k \neq j\}|}{|I_{m,n}|} < \eta;$$
%\item[\rm (ii)] There is some $D \subseteq \Z / k\Z$ such that $$\frac{|\{i \leq |h_m|: [i]_k \in D\} \Delta I_{l,m} |}{|I_{l,m}|} < \epsilon$$\end{enumerate}
%Thus condition (II) of this theorem is satisfied.

%We now show that condition (II) of this theorem implies condition (I).  
Suppose now that condition (II) holds.  For all $l \in \N$ and all $\epsilon >0$, produce $k_{l, \epsilon}$, and $N_{l, \epsilon}$ according to condition (II).  
%We now close the set $\{k_{l, \epsilon}: l \in \N, \epsilon>0\}$ under factors. 
Let $$K = \{k \in \N : k | k_{l, \epsilon} \textnormal{ for some $l \in \N$ and $\epsilon>0$}\}.$$ 

It is clear that $K$ is closed under factors.  We leave it to the reader to show that $K$ is infinite by showing that if $l \in \N$ and $\epsilon<1$, then $k_{l, \epsilon} \geq h_l$.  
%Indeed, let $l \in \N$ and $0< \epsilon< 1$ and consider $k_{l, \epsilon}$.  Let $m$ be large enough that we can choose $D \subseteq \Z/k_{l, \epsilon}\Z$ such that $$\frac{|\{i < h_m: [i]_{k_{l, \epsilon}} \in D\} \Delta I_{l,m} |}{|I_{l,m}|} < \epsilon.$$ Suppose, towards a contradiction, that $k_{l, \epsilon} < h_l$.  Consider any $i \in I_{l,m}$.  This means that $T^i(B_m)$ is one of the levels of the stage-$m$ tower that is at the base of one of the $l$-blocks.  First notice that since $k_{l, \epsilon}< h_l$, $T^{i + k_{l, \epsilon}} (B_m)$ will be one of the levels--but not the base--of the same $l$-block in the stage-$m$ tower.  It follows that $i+k_{l, \epsilon}\not\in I_{l,m}$. Note that $[i+k_{l,\epsilon}]_{k_{l,\epsilon}}=[i]_{k_{l,\epsilon}}$. Thus $[i]_{k_{l,\epsilon}}\in D$ if and only if $[i+k_{l,\epsilon}]_{k_{l,\epsilon}}\in D$. It follows that $$ |\{i<h_m : [i]_{k_{l,\epsilon}}\in D\}\setminus I_{l,m}|\geq |\{i<h_m : [i]_{k_{l,\epsilon}}\in D\}\cap I_{l,m}|. $$ Thus, $$\frac{|\{i < h_m: [i]_{k_{l, \epsilon}} \in D\} \Delta I_{l,m} |}{|I_{l,m}|} \geq 1,$$ which is a contradiction.  Thus $k_{l, \epsilon} \geq h_l$ and we can conclude that $K$ is infinite.
  
Now, consider $\mathcal{O}_K$.  We will prove that $T$ is isomorphic to $\mathcal{O}_K$ by showing that conditions (IIa) and (IIb) of Theorem \ref{isomorphictothisodometer} hold.  First, let $k \in K$.  Choose $l \in N$ and $\epsilon>0$ such that $k | k_{l, \epsilon}$.  We chose $k_{l, \epsilon}$ using condition (II) of this theorem.  Theorem \ref{finitefactor1} guarantees that that $T$ factors onto $\Z / k_{l, \epsilon}\Z$.  Therefore, $T$ must also factor onto $\Z/k\Z$.  Now Theorem \ref{finitefactor1} guarantees that condition (IIa) of Theorem \ref{isomorphictothisodometer} holds.  Condition (IIb) of Theorem \ref{isomorphictothisodometer} follows immediately from our assumption that condition (II) of this theorem holds and our choice of $K$.
\end{proof}

Before closing we consider an example of a rank-one transformation that factors onto an odometer but is not isomorphic to any odometer. 

\smallskip

\noindent {\bf Example.} Let $T$ be the rank-one transformation corresponding to the symbolic definition $v_0=0$ and 
$$ v_{n+1}=v_nv_n1^{2^{n+1}}v_nv_n. $$
Then the length of $v_n$, or equivalently the height $h_n$ of the stage-$n$ tower, is $2^n(2^{n+1}-1)$. Using Theorem~\ref{finitefactor1} it is easy to verify that $T$ has all powers of $2$ as finite factors. Thus $T$ factors onto the dyadic odometer. As n1oted by the referee, ergodicity of the dyadic powers and non-ergodicity of the odd powers follows from \cite[Theorem H]{D19}. An argument using Theorem~\ref{finitefactor1} also shows that $T$ does not have any other factors. Indeed, suppose $T$ has an odd finite factor $a$. If no multiples of $a$ are of the form $2^k-1$ for any $k$, then the condition in Theorem~\ref{finitefactor1} fails, since the elements of $I_{m,n}$ come in pairs, with a difference $h_m=2^m(2^{m+1}-1)$ between them. On the other hand, suppose $a$ has a multiple of the form $2^{m+1}-1$ for some $m$. Then note that the elements of $I_{m,n}$ come in quadruples, with the sequence of differences $h_m, 2^{m+1}, h_m$ in between them. This implies also that at least half of the indices of $I_{m,n}$ disagree on the congruence class mod $a$, and thus the condition in Theorem~\ref{finitefactor1} fails. Therefore the maximal odometer factor of $T$ is the dyadic odometer. Finally, a similar argument shows that condition (IIb) of Theorem~\ref{isomorphictothisodometer} fails. Consequently $T$ is not isomorphic to the dyadic odometer. In conclusion, $T$ is not isomorphic to any odometer.

\end{document}